\documentclass[a4paper]{article}
\usepackage{amsmath,amsfonts,amssymb,amsthm}
\usepackage[english]{babel}
\usepackage{color}
\usepackage{tikz}

\newcommand{\cB}{{\mathcal B}}
\newcommand{\cD}{{\mathcal D}}

\newcommand{\cH}{{\mathcal H}}

\newcommand{\cO}{{\mathcal O}}

\newcommand{\cS}{{\mathcal S}}

\newcommand{\myspace}{\qquad\qquad\qquad}

\newtheorem{theorem}{Theorem}[section]

\newtheorem{proposition}[theorem]{Proposition}
\newtheorem{remark}[theorem]{Remark}
\newtheorem{remarks}[theorem]{Remarks}

\newtheorem{problem}[theorem]{Problem}
\numberwithin{equation}{section}

\begin{document}

\title{\textbf{{\Large Exponential decay properties of a mathematical model
for a certain fluid-structure interaction{\thanks{
The research of G.~Avalos was partially supported by the NSF~Grants
DMS-0908476 and DMS-1211232. The research of F.~Bucci was partially
supported by the Italian MIUR under the PRIN~2009KNZ5FK Project (\emph{\
Metodi di viscosit\`a, geometrici e di controllo per modelli diffusivi
nonlineari}), by the GDRE (Groupement De Recherche Europ\'een) CONEDP (\emph{%
\ Control of PDEs}), and also by the Universit\`a degli Studi di Firenze
under the Project \emph{Calcolo delle variazioni e teoria del controllo}.} }}%
}}
\author{ George Avalos \\
{\normalsize {University of Nebraska-Lincoln}}\\
{\small {Lincoln NE, U.S.A.}}\\
{\small \texttt{gavalos@math.unl.edu}} \and Francesca Bucci\thanks{%
F.~Bucci is member of the Gruppo Nazionale per l'Analisi Matematica, la
Probabilit\`a e le loro Applicazioni (GNAMPA) of the Istituto Nazionale di
Alta Matematica (INdAM).} \\
{\small {Universit\`a degli Studi di Firenze}}\\
{\small {Firenze, ITALY}}\\
{\small \texttt{francesca.bucci@unifi.it} }}

\date{}

\maketitle

\begin{abstract}
In this work, we derive a result of exponential stability for a coupled system of partial differential equations (PDEs) which governs a certain fluid-structure interaction. 
In particular, a three-dimensional Stokes flow interacts across a boundary interface 
with a two-dimensional mechanical plate equation. 
In the case that the PDE plate component is rotational inertia-free, one will have that 
solutions of this fluid-structure PDE system exhibit an exponential rate of decay. 
By way of proving this decay, an estimate is obtained for the resolvent of the associated 
semigroup generator, an estimate which is uniform for frequency domain values along $i\mathbb{R}$.
Subsequently, we proceed to discuss relevant point control and boundary control scenarios for 
this fluid-structure PDE model, with an ultimate view to optimal control studies on both
finite and infinite horizon. 
(Because of said exponential stability result, optimal control of the PDE on time interval 
$(0,\infty)$ becomes a reasonable problem for contemplation.)  
\end{abstract}


\section*{Introduction}
In this work, we undertake a stability analysis of a certain partial
differential equation (PDE) system, that is \eqref{e:pde-model}-\eqref{ic} below, 
which has been previously studied in \cite{chu-ryz} and \cite{chambolle},
among other works, inasmuch as it simultaneously constitutes a 
mathematically interesting and physically relevant model of a
fluid-structure (F-S) dynamics. 
This PDE model comprises a Stokes flow, evolving within a three-dimensional cavity $\mathcal{O}$, coupled via a boundary interface, to a two dimensional Euler-Bernoulli or Kirchhoff plate which
displaces upon a bounded open set $\Omega$, which is taken to be a portion
of the cavity boundary $\partial \mathcal{O}$. 
Our main result here (Theorem~\ref{t:main} below) is the derivation of \emph{exponential decay} rates for the composite fluid-structure dynamics, in the case that the Euler-Bernoulli plate PDE
model -- i.e. the one corresponding to ``rotational inertia'' parameter $\rho =0$ in \eqref{1} -- is used to describe the mechanical displacements along $\Omega$. 

In point of fact, this stability result was originally given in \cite{chu-ryz}; 
the real novelty in the present work lies in the method of proof: whereas in \cite{chu-ryz}
the exponential decay of the given fluid-structure dynamics is obtained via a Lyapunov
functional approach, with the authors of \cite{chu-ryz} operating strictly
within the \emph{time domain}, the present work is centered upon working
instead in the \emph{frequency domain}. 
In particular, we work to attain a uniform estimate for the resolvent operator of the generator 
of the associated fluid-structure semigroup, as it assumes values along the
imaginary axis. With such resolvent estimate in hand, we can then appeal to
a well-known resolvent criterion of Pr\"{u}ss (posted as Theorem~\ref{t:pruess} here),
so as to ultimately infer exponential decay. 
The virtue of the frequency domain approach which is employed here, is that it can eventually 
be adapted so as to treat the case $\rho >0$ (Kirchhoff plate). 

Indeed, the frequency domain methodology outlined here is invoked and refined in \cite%
{avalos-bucci-arxiv2013}, so as to provide rational decay rates for
Stokes-Kirchhoff plate dynamics. 
(The higher topology for the mechanical velocity component in Kirchhoff plate equation -- viz., 
$H^{1}(\Omega )$ for the Kirchhoff plate, as opposed to $L^{2}(\Omega )$ for Euler-Bernoulli --
prevents the attainment of exponential decay in \cite{avalos-bucci-arxiv2013}. 
Hence, weaker polynomial rates of decay are established for $\rho >0$.) 
We should also state that our estimate of said resolvent on the imaginary axis
is direct and explicit, in the style of what was performed in \cite{avalos-trigg02};
previous exponential stability works which are geared so as to eventually invoke said resolvent criterion of Pr\"{u}ss, tend to obtain the requisite resolvent estimates via an argument by contradiction (see, e.g., \cite{liu2}).

\medskip
In addition, in the final Section of this work we offer some insight into a
further analysis which is needed to pursue solvability of
natural/appropriate optimal control problems (with quadratic functionals)
associated with the PDE system under investigation. We note that a full
understanding of the stability properties of this F-S interaction is not
only of intrinsic interest, but indeed a prerequisite step in the study of
optimal control problems over an \emph{infinite time} horizon. In this
respect, the uniform (exponential) stability result established for the
composite PDE system \eqref{e:pde-model} in the presence of an elastic
equation of Euler-Bernoulli type 
ensures that  \emph{both finite and infinite} time horizon problems are equally 
valid objects of investigation. 
Instead, when the elastic equation is of Kirchhoff type, the (rational)
decay rates of solutions shown in \cite{avalos-bucci-arxiv2013} will prevent
us from taking into consideration optimal control problems on an infinite time interval.

We also point out here that the presence of boundary or point control
actions into the model, will necessitate a careful technical analysis of the
regularity properties of the \emph{kernel} term $e^{At}B$ of the (so called)
``input-to-state map'' for the abstract controlled equation $y^{\prime }=Ay+Bg$ 
corresponding to the controlled boundary value problem. 
This is particularly so, inasmuch as, in relevant applications the control operator $B$ 
is (intrinsically) not bounded from the state space into the control space.

In turn, as is well known, sharp PDE regularity estimates for the solutions to the
``free'' (or uncontrolled) system should be instrumental in bringing about the 
sought regularity properties of the operator $e^{At}B$. 
It is only by having at hand such information on the sharp regularity of the 
fluid-structure PDE under present study, that one can verify whether the recent 
results on the LQ-problem and Riccati equations for abstract dynamics inspired by 
and tailored for coupled PDE systems of hyperbolic/parabolic type (such as 
\cite{abl-2005,abl-2013}) are applicable, or whether novel theories need to be devised. 
A brief description of a couple of relevant scenarios for the placement of control
functions in the model is given, along with some remarks about the technical
challenges which are expected.

\section{The PDE model, statement of the main result}

In what follows, the same geometry which prevailed in \cite{chu-ryz} will
obtain here. Namely: (fluid) domain $\mathcal{O}\subset \mathbb{R}^{3}$ will
be bounded with sufficiently smooth boundary $\partial \mathcal{O}$.
Moreover, $\partial \mathcal{O}=\overline{S}\cup \overline{\Omega}$, with 
$S \cap \Omega=\emptyset$, $\Omega$ being a flat portion of $\partial \mathcal{O}$.
In particular, $\partial \mathcal{O}$ has the following specific spatial configuration: 
\begin{equation*}
\Omega \subset \left\{ x=(x_{1,}x_{2},0)\right\} \,,\quad S\subset \left\{
x=(x_{1,}x_{2},x_{3}):x_{3}\leq 0\right\} \,.
\end{equation*}
So if $\nu (x)$ denotes the unit normal vector to $\partial \mathcal{O}$,
pointing outward, then 
\begin{equation}  \label{normal}
\nu|_{\Omega }=\left[ 0,0,1\right]\,. 
\end{equation}
See, e.g., the picture below.

\smallskip 
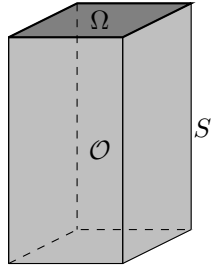
\begin{center}
\begin{tikzpicture}[scale=1.5]

\draw [fill=lightgray,lightgray] (.6,.3) rectangle (1.6,2.3);

\draw [fill=lightgray,lightgray] (0,0) rectangle (1,2);

\draw[fill=lightgray,lightgray] (0,0)--(1,0)--(1.6,.3)--(.6,.3);

\draw [thick, fill=gray] (0,2)--(1,2)--(1.6,2.3)--(.6,2.3)--(0,2);

\node at (.8,1) {$\mathcal{O}$};

\node at (1.7,1.2) {$S$};

\node at (.8,2.15) {$\Omega$};

\draw  (0,0)--(0,2)--(1,2)--(1,0)--(0,0);

\draw (0,2)--(.6,2.3)--(1.6,2.3)--(1,2);

\draw (1.6,.3)--(1,0);

\draw [dashed] (0,0)--(.6,.3)--(1.6,.3);

\draw (0,0)--(0,2);

\draw (1,0)--(1,2);

\draw [dashed](.6,2.3)--(.6,.3);

\draw (1.6,.3)--(1.6,2.3);

\end{tikzpicture}
\end{center}

\medskip
On such geometry, the PDE model is as as follows, with rotational inertia parameter
$\rho \ge 0$, and in solution variables $u(x,t)=[u^{1}(x,t),u^{2}(x,t),u^{3}(x,t)]$ 
and $[w(x,t),w_{t}(x,)]$:

\begin{subequations} \label{e:pde-model}
\begin{align}
& u_{t}-\Delta u+\nabla p=0 & & \text{in}\;\mathcal{O}\times (0,T)  \label{3}
\\
& \mathrm{div}(u)=0 & & \text{in}\;\mathcal{O}\times (0,T)  \label{4} \\
& u=0\;\; & & \text{on }\,S  \label{5-s} \\
& u=[u^{1},u^{2},u^{3}]=[0,0,w_{t}] & & \text{on }\,\Omega \,,
\label{5-omega} \\
& w_{tt}-\rho \Delta w_{tt}+\Delta ^{2}w=p|_{\Omega } & & \text{in}\;\Omega
\times (0,T)  \label{1} \\
& w=\frac{\partial w}{\partial \nu }=0 & & \text{on}\;\partial \Omega
\label{2}
\end{align}
with initial conditions 
\end{subequations}
\begin{equation}
[u(0),w(0),w_{t}(0)]=[u_{0},w_{0},w_{1}]\in \mathbf{H}_\rho\,.
\label{ic}
\end{equation}
Here, the space of initial data $\mathbf{H}_\rho$ is defined as follows: 
Let the (fluid) space $\mathcal{H}_{fluid}$ ($\mathcal{H}_f$, in short) be defined by
\begin{equation}
\mathcal{H}_f:=\left\{ f\in \mathbf{L}^{2}(\mathcal{O}):
\mathrm{div}(f)=0\text{; }\left. f\cdot \nu \right\vert _{S}=0\right\}\,,
\label{H_f}
\end{equation}
and let 
\begin{equation} \label{velocity}
V_{\rho }=
\begin{cases}
L^2(\Omega)/\mathbb{R} & \text{if $\rho =0$} 
\\[1mm]
H_0^1(\Omega )\cap L^2(\Omega)/\mathbb{R} & \text{if $\rho>0$.}
\end{cases}
\end{equation}
Therewith, we then set 
\begin{equation} \label{energy}
\begin{split}
\mathbf{H}_{\rho} &=\Big\{\big[f,h_0,h_1\big] \in 
\mathcal{H}_f \times \big[H_0^2(\Omega )\cap L^{2}(\Omega)/\mathbb{R}\big] 
\times V_{\rho }\,,
\nonumber \\
& \myspace \text{with }\; f\cdot \nu|_{\Omega}=[0,0,f^3]\cdot [0,0,1]=h_1\Big\}\,.  
\end{split}
\end{equation}
In this paper, we shall focus on the case $\rho =0$.

\smallskip
In addition: By way of constructing an abstract operator 
$A_{\rho}:\cD(A_{\rho })\subset \mathbf{H}_\rho\rightarrow \mathbf{H}_\rho$ which describes
the PDE dynamics \eqref{e:pde-model}-\eqref{ic}, we denote 
$A_D:L^2(\Omega)\rightarrow L^2(\Omega)$ by 
\begin{equation} \label{dirichlet}
A_Dg=-\Delta g\,, \qquad \cD(A_{D})=H^2(\Omega )\cap H_0^1(\Omega)\,.
\end{equation}
If we subsequently make the denotation for all $\rho \ge 0$
\begin{equation} \label{P}
P_\rho=I+\rho A_{D}\,, \quad \cD(P_\rho)=
\begin{cases}
L^2(\Omega ) & \text{if $\rho =0$} 
\\ 
\cD(A_D) & \text{if $\rho =0$} 
\end{cases},
\end{equation}
then the mechanical PDE component \eqref{1} of the system \eqref{e:pde-model} can be written as 
\begin{equation*}
P_\rho w_{tt}+\Delta ^{2}w=\left. p\right\vert _{\Omega } \quad \text{on $(0,T)$.}
\end{equation*}

Using the fact from \cite{grisvard} that
\begin{equation*}
\cD(P_{\rho }^{1/2})=
\begin{cases}
L^2(\Omega) & \text{if $\rho =0$} 
\\ 
H_0^1(\Omega ) & \text{if $\rho >0$} 
\end{cases},
\end{equation*}
then we can endow the Hilbert space $\mathbf{H}_{\rho }$ with the norm-inducing
inner product
\begin{equation*}
\big([\mu_0,\omega_1,\omega_2],
[\tilde{\mu}_{0},\tilde{\omega}_1,\tilde{\omega}_2]\big)_{\mathbf{H}_{\rho}}=
(\mu_0,\tilde{\mu}_{0})_{\cO}+(\Delta \omega_{1},\Delta \tilde{\omega}_{1})_{\Omega }
+(P_{\rho }^{1/2}\omega_{2},P_{\rho }^{1/2}\tilde{\omega}_{2})_{\Omega }\,,
\end{equation*}
where $(\cdot,\cdot)_{\mathcal{O}}$ and $(\cdot,\cdot)_{\Omega }$ are
the $L^{2}$-inner products on their respective geometries.

\smallskip
Moreover, as was done in \cite{avalos-clark} and \cite[Lemma~1.1]{avalos-bucci-arxiv2013}, so as to eliminate the pressure term $p$ in \eqref{e:pde-model}-\eqref{ic} (see also \cite{dvorak} for an analogous elimination for a different fluid-structure PDE model), we recognize the
pressure term as the solution of the following BVP, pointwise in time: 
\begin{equation} \label{e:bvp-for-pressure}
\begin{cases}
\Delta p=0 & \text{in }\;\mathcal{O} 
\\[1mm] 
\frac{\partial p}{\partial \nu }=\Delta u\cdot \nu \big|_{S} & \text{on $S$}
\\[1mm] 
\frac{\partial p}{\partial \nu }+p=\Delta ^{2}w+\Delta u^{3}\big|_{\Omega }
& \text{on $\Omega $.}
\end{cases}
\end{equation}

To ``solve'' for the pressure term, we then invoke appropriate (`Robin') maps $R_{\rho}$ and $\tilde{R}_{\rho}$ defined as follows:
\begin{equation*}
\begin{aligned} 
R_{\rho }g &= f\Longleftrightarrow \Big\{ \Delta f=0\text{ \
in }\mathcal{O}\,, \; \frac{\partial f}{\partial \nu }+P_{\rho}^{-1}f=g\text{ \ on } \Omega\,, \; \frac{\partial f}{\partial \nu}=0\text{\ on }S\Big\}\,; 
\\ 
\tilde{R}_{\rho }g &=f\Longleftrightarrow
\Big\{ \Delta f=0\text{ \ in } \mathcal{O}\,, \; \frac{\partial f}{\partial
\nu }+P_{\rho }^{-1}f=0\text{\ on }\Omega\,, \; \frac{\partial f}{\partial
\nu }=g\text{ on } S\Big\}\,. 
\end{aligned}
\end{equation*}
Therewith, we have that for all real $s$,
\begin{equation}
R_{\rho }\in \mathcal{L}\big(H^{s}(\Omega ),H^{s+3/2}(\mathcal{O})\big)\,; 
\quad 
\tilde{R}_{\rho }\in \mathcal{L}\big(H^{s}(S),H^{s+3/2}(\mathcal{O})\big)\,.
\end{equation}
(We are also using implicity the fact that $P_{\rho }^{-1}$ is positive
definite and self-adjoint on $\Omega$.)
Consequently, the pressure variable $p(t)$, as necessarily the solution of 
\eqref{e:bvp-for-pressure} -- that is an appropriate {\em harmonic extension} 
from the boundary of $\cO$ into the interior -- can be written pointwise in 
time as 
\begin{equation} \label{p_2}
p(t)=G_{\rho ,1}(w(t))+G_{\rho ,2}(u(t))\,,  
\end{equation}
where 
\begin{subequations} \label{G}
\begin{align} 
G_{\rho ,1}(w) &=R_{\rho }(P_{\rho }^{-1}\Delta ^{2}w)\,; 
\\
G_{\rho ,2}(u) &=R_{\rho }(\left. \Delta u^{3}\right\vert _{\Omega })+
\tilde{R}_{\rho }(\left. \Delta u\cdot \nu \right\vert_{S})\,.  
\end{align}
\end{subequations}

These relations suggest the following choice for the generator 
$A_{\rho }:\mathbf{H}_{\rho }\rightarrow \mathbf{H}_{\rho }$. 
We set

\begin{equation} \label{domain}
A_{\rho }\equiv 
\begin{bmatrix}
\Delta -\nabla G_{\rho ,2} & -\nabla G_{\rho ,1} & 0 \\ 
0 & 0 & I \\ 
P_{\rho }^{-1}G_{\rho ,2}\big|_{\Omega } & -P_{\rho }^{-1}\Delta
^{2}+P_{\rho }^{-1}G_{\rho ,1}\big|_{\Omega } & 0
\end{bmatrix}  
\end{equation}
with domain 
\begin{equation} \label{e:domain-of-generator}
\begin{split}
{\cD}(A_{\rho })& =\Big\{\big[u,w_{1},w_{2}\big]\in \mathbf{H}_{\rho}:
\; u\in \mathbf{H}^2(\mathcal{O})\,; 
\quad w_1\in {\cS}_{\rho}\,,\; w_2\in H_0^2(\Omega )\,,
\\[1mm]
& \myspace u=0\; \text{on}\;S\,,\quad u=(0,0,w_2)\; \text{on}\;\Omega\,\Big\}\,,
\end{split}
\end{equation}
where the mechanical displacement space, denoted by $\cS_\rho$, changes with 
$\rho$ as follows: 
\begin{equation*}
\cS_{\rho }:=
\begin{cases}
H^{4}(\Omega )\cap H_{0}^{2}(\Omega ) & \rho =0 
\\ 
H^{3}(\Omega )\cap H_{0}^{2}(\Omega ) & \rho >0\,.
\end{cases}
\end{equation*}

\smallskip
We note also, from the definition of $\cD(A_{\rho })$ that 
$[u,w_{1},w_{2}] \in \cD(A_{\rho })$ implies $\Delta u\in \mathbf{L}^2(\cO)$ and 
$\text{div}\Delta u=0$. 
Consequently, from elementary Stokes Theory (see, e.g., \cite[Proposition~1.4, p.~5]{const}, 
we have
\begin{equation*}
\|\Delta u\cdot \nu\|_{H^{-1/2}(\partial \cO)} \le C\|u\|_{\mathbf{H}^{2}(\cO)}
\le C \|[u,w_{1},w_{2}]\|_{\cD(A_{\rho })}  
\end{equation*}
and so associated pressure $\pi_0$ satisfies
\begin{equation}\label{p_reg}
\pi_0 \equiv G_{\rho,1}(w_{1})+G_{\rho ,2}(u)\in H^{1}(\mathcal{O})\,.  
\end{equation}

\medskip

\begin{remarks}
\begin{rm}
\textbf{(i) (Well-posedness) } 
Well-posedness of the (\emph{linear}) coupled system \eqref{e:pde-model}-\eqref{ic} when 
$\rho =0$ -- namely, when the elastic equation is of Euler-Bernoulli type, of specific concern 
in the present investigation --, was originally established in \cite{chu-ryz}, by
using Galerkin approximations. 
An alternative proof of well-posedness which encompasses both cases $\rho =0$
and $\rho >0$ has been recently given in \cite{avalos-clark}. 
It is important to emphasize that the proof appeals to the Lumer-Phillips Theorem
within classical semigroup theory, and yet also utilizes in a crucial and
nontrivial way the Babu\v{s}ka-Brezzi~Theorem (see, e.g., \cite[p.~116]{kesavan}). 
The corresponding statement is given below. 

\begin{theorem}[\protect\cite{avalos-clark}] \label{t:well} 
The operator $A_{\rho }:\mathbf{H}_{\rho }\rightarrow \mathbf{H}_{\rho }$ 
defined by \eqref{domain}-\eqref{e:domain-of-generator}
generates a $C_{0}$-semigroup of contractions $\left\{ e^{A_{\rho}t}\right\}_{t\ge 0}$ 
on $\mathbf{H}_{\rho }$. 
Thus, for any $[u_{0},w_{0},w_{1}]\in \mathbf{H}_{\rho }$, the (unique) weak solution to
the initial/boundary value problem \eqref{e:pde-model}-\eqref{ic} is given by 
\begin{equation} \label{e:semigroup}
\left[ 
\begin{array}{c}
u(t) \\ 
w(t) \\ 
w_{t}(t)
\end{array}
\right] =e^{A_{\rho }t}\left[ 
\begin{array}{c}
u_{0} \\ 
w_{0} \\ 
w_{1}
\end{array}
\right] \in C([0,T];\mathbf{H}_{\rho })\,.  
\end{equation}
\end{theorem}

\smallskip 
\noindent 
\textbf{(ii) (Decay rates) } 
To the authors' knowledge, the stability properties of solutions to the \emph{linear} model 
\eqref{e:pde-model} (again, when $\rho =0$) have been explored in \cite{chu-ryz}, along the analysis of the long-term behaviour of a \emph{nonlinear} coupled dynamics, comprising a $3$D linearized Navier-Stokes system for the fluid velocity field in a bounded domain, and a 
\emph{nonlinear} elastic plate equation for the transversal displacement of a flat flexible part 
of the boundary. 
Among the various results established in \cite{chu-ryz}, primarily pertaining to the nonlinear model, exponential stability of the linear dynamics is attained by using Lyapunov function arguments; see \cite[Section~3]{chu-ryz}.  
\end{rm}
\end{remarks}

\medskip 
We aim here at presenting a different proof of exponential
stability, based instead on a (by now classical) resolvent criterion by Pr\"{u}ss; 
see Theorem~\ref{t:pruess} in the next Section. 
The adoption of a ``frequency domain'' approach -- in contrast with the more commonly invoked 
``time domain'' analysis -- is not only of intrinsic interest, but it also proves to be very effective in order to establish the decay rates of solutions, \emph{even when exponential stability fails}. 
Indeed, in the case $\rho >0$, the very same frequency domain approach enables us to establish
that the energy of strong solutions decays at the rate of $O(1/t)$ , as 
$t\rightarrow +\infty$; see \cite{avalos-bucci-arxiv2013}.

The main result of the present work is stated below.


\begin{theorem}[\textbf{Exponential decay rates}] \label{t:main} 
Let the rotational inertia parameter $\rho =0$ in \eqref{1}.
Then all finite energy solutions of \eqref{e:pde-model}-\eqref{ic} decay at
an exponential rate. Namely, there exist constants $\omega >0$ and $M\ge 1$
such that for arbitrary initial data $[u_{0},w_{0},w_{1}]\in \mathbf{H_{0}}$, 
the corresponding solutions $[u,w,w_{t}]$ of \eqref{e:pde-model}-\eqref{ic}
satisfy 
\begin{equation*}
\|[u(t),w(t),w_{t}(t)]\|_{\mathbf{H}_{0}}\le
M\,e^{-\omega t}\,\|[u_{0},w_{0},w_{1}]\|_{\mathbf{H}_0}\,.
\end{equation*}
\end{theorem}


\section{Exponential Stability}

To show that the semigroup defined by \eqref{e:semigroup} is exponentially
stable, we appeal to a celebrated result of semigroup theory which we recall
explicitly for the reader's convenience.

\begin{theorem}[\protect\cite{pruess}] \label{t:pruess} 
Let $(T(t))_{t\ge 0}$ be a $C_{0}$-semigroup on a Hilbert space $H$ with generator $A$, 
such that $i\mathbb{R}\subset \varrho (A)$. 
Then, the following are equivalent: 
\begin{align*}
(i)\quad & \exists C>0: \; \|R(is;A)\|\le C \quad \forall s\in \mathbb{R}\,;
\\[1mm]
(ii)\quad & \exists \omega>0\,, \; M\ge 1\,: \; \|T(t)\|\le M e^{-\omega t}
\qquad t\ge 0\,.
\end{align*}
\end{theorem}
In order to invoke the above resolvent criterion, we need, as a preliminary
step, to show that the imaginary axis belongs to the resolvent set of the
dynamics operator $A$. This property cannot be freely taken for granted: in
the context of other fluid-structure interactions, it is known that certain
geometrical configurations will give rise to eigenvalues on the imaginary
axis; see, e.g., \cite{dvorak} and \cite{A-T} (and also \cite{avalos-trigg02}, 
where examples of ``non-pathological geometries'' are given).


\subsection{Preliminary step: Spectral Analysis}

Here, we limit ourselves to show that $\lambda =0$ belongs to the resolvent
set $\varrho(A)$; in other words, the resolvent operator is boundedly invertible on the
state space $\mathbf{H}_\rho$. The reader is referred to 
\cite[Section~2]{avalos-bucci-arxiv2013} for a detailed analysis and proof of the fact that
the spectrum has empty intersection with the whole imaginary axis, in the
more challenging case $\rho >0$. 
The arguments used therein can be easily adapted to the case $\rho =0$.

As the parameter $\rho$ equals $0$ throughout, in order to simplify the
notation we set $\mathbf{H_0}=:\mathbf{H}$, as well as $A_0=:A$. 
(We note that $P_\rho$ reduces coincides with the identity operator $I$ throughout.)

\begin{proposition} \label{p:zero-in-the-resolventset} 
The generator $A:{\mathcal{D}}(A)\subset 
\mathbf{H}\rightarrow \mathbf{H}$ is boundedly invertible on $\mathbf{H}$.
Namely, $\lambda =0$ is in the resolvent set of $A$.
\end{proposition}

\begin{proof}
Given data $[u^{\ast },w_{1}^{\ast },w_{2}^{\ast }]\in 
\mathbf{H}$, we look for $[u,w_{1},w_{2}]\in {\mathcal{D}}(A)$ which solves 
\begin{equation}
A\left[ 
\begin{array}{c}
u \\ 
w_{1} \\ 
w_{2}%
\end{array}%
\right] =\left[ 
\begin{array}{c}
u^{\ast } \\ 
w_{1}^{\ast } \\ 
w_{2}^{\ast }%
\end{array}%
\right] \,.  \label{resolve}
\end{equation}%
To this end, we must seek for $[u,w_{1},w_{2}]$ in ${\mathcal{D}}(A)$ and
associated pressure $\pi_{0}\in H^{1}({\mathcal{O}})$ which solve 
\begin{subequations}
\label{e:explicit-resolvent-eq}
\begin{align}
& \Delta u-\nabla \pi _{0}=u^{\ast } & & \text{in ${\mathcal{O}}$}
\label{r4} \\
& \text{div}(u)=0 & & \text{in ${\mathcal{O}}$}  
\label{r5} \\
& u=0 & & \text{on $S$}  
\label{r6-ii} \\
& u=(0,0,w_{2}) & & \text{on $\Omega $}  
\label{r6-i} \\
& w_{2}=w_{1}^{\ast } & & \text{in $\Omega $}  
\label{r1} \\
& \Delta ^{2}w_{1}-\pi _{0}\big|_{\Omega }=-w_{2}^{\ast } & & \text{in $\Omega $}  
\label{r2} \\
& w_{1}=\frac{\partial w_{1}}{\partial \nu }=0 & & \text{on $\partial \Omega$.}
\label{r3}
\end{align}
Moreover, we must justify that the pressure variable $\pi _{0}$ above is
given by the expression 
\end{subequations}
\begin{equation} \label{pi_0}
\pi_{0}=G_{1}(w_{1})+G_{2}(u)\,,  
\end{equation}
where we simply denoted by $G_{i}$, $i=1,2$, the operators $G_{0,i}$ defined in \eqref{G} 
(in line with the appearance of $A$ in \eqref{domain}).

\medskip \noindent 
\emph{1. The Plate Velocity. } From \eqref{r1}, the velocity component $w_2$ is immediately resolved.

\medskip \noindent 
\emph{2. The Fluid Velocity. } 
We next consider the Stokes system \eqref{r4}--\eqref{r6-i}. 
From \eqref{r1} and \eqref{r6-ii}-\eqref{r6-i} we have $u|_{\partial {\mathcal{O}}}$ satisfies 
\begin{equation}  \label{cond}
\int_{\partial\mathcal{O}} u\cdot \nu\, d\sigma
=\int_{\Omega } [0,0,u^{3}] \cdot \nu \, d\Omega =\int_{\Omega}w_2\,d\Omega
=\int_{\Omega}w_1^*\,d\Omega =0\,,
\end{equation}
where the last equality follows by the membership $[u^*,w_1^*,w_2^*]\in \mathbf{H}$. 
Since this compatibility condition
is satisfied and data $\{u^*,w_1^*\} \in \mathbf{L}^2({\mathcal{O}})\times
H_0^2(\Omega)$, we can find a unique (fluid and pressure) pair 
$(u,q_0)\in [\mathbf{H}^2({\mathcal{O}})\cap \cH_f] \times 
\mathbf{H}^1(\mathcal{O})/\mathbb{R}$ which solve 
\begin{subequations}
\begin{align}
& \Delta u-\nabla q_0=u^* & \text{in ${\mathcal{O}}$}  \label{s1} \\
& \text{div}(u)=0 & \text{in ${\mathcal{O}}$}  \label{s2} \\
& u=0 & \text{on $S$}  \label{s3-ii} \\
& u=(0,0,w_1^*) & \text{on $\Omega$}.  \label{s3-i}
\end{align}
\end{subequations}
Moreover, one has the estimate 
\begin{equation}  \label{stokes}
\| u\|_{\mathbf{H}^2(\cO)\cap \mathcal{H}_{f}} +\|q_0\|_{\mathbf{H}^1({\mathcal{O}})/\mathbb{R}} \le C\big[ \|u^*\|_{\mathcal{H}_f}
+\|w_1^*\|_{H_0^2(\Omega)}\big]
\end{equation}
(see e.g., \cite[Proposition 2.3, p.~25]{temam}).

\medskip \noindent 
\emph{3. The Mechanical Displacement. } 
Subsequently, we consider the plate component boundary value problem (BVP) 
\eqref{r2}-\eqref{r3}. 
By ellipticity and elliptic regularity (see \cite{L-M}) there exists a solution 
$\hat{w}_{1}\in H^{4}(\Omega )\cap H_{0}^{2}(\Omega )$ to the problem 
\begin{equation}
\begin{cases}
\Delta ^{2}\hat{w}_{1}=q_{0}|_{\Omega }-w_{2}^{\ast } & \text{in $\Omega $}
\nonumber \\[1mm] 
\hat{w}_{1}=\displaystyle\frac{\partial \hat{w}_{1}}{\partial \nu }=0 & 
\text{on $\partial \Omega $}
\end{cases}
\label{p_1}
\end{equation}
where $q_{0}$ is the pressure variable in \eqref{s1}. Moreover, we have the
estimate 
\begin{align}
\Vert \hat{w}_{1}\Vert _{H^{4}(\Omega )\cap H_{0}^{2}(\Omega )}& \leq
C\,\Vert q_{0}|_{\Omega }+w_{2}^{\ast }\Vert _{L^{2}(\Omega )}  \notag \\
& \leq C\,\Vert q_{0}|_{\Omega }\Vert _{H^{1/2}(\Omega )}+\Vert w_{2}^{\ast
}\Vert _{L^{2}(\Omega )}  
\notag \\
& \leq C\,\Vert [u^*,w_1^*,w_2^*]\Vert_{\mathbf{H}}  
\label{plate_hat}
\end{align}
(in the last inequality we have also invoked Sobolev Trace Theory and \eqref{stokes}).

Now if, as in \cite{chu-ryz}, we let $\mathbb{P}$ denote the orthogonal
projection of $H_0^2(\Omega)$ onto $H_0^2(\Omega)\cap L^2(\Omega)/\mathbb{R}$
-- orthogonal with respect to the inner product 
$[\omega,\tilde{\omega}]\rightarrow (\Delta \omega,\Delta \tilde{\omega})_{\Omega }$ --, then one
can readily show that its orthogonal complement $I-\mathbb{P}$ can be
characterized as 
\begin{equation}  \label{constant}
\begin{split}
& (I-\mathbb{P})H_0^2(\Omega)=\text{Span}\{\varphi \}\,, \quad \text{where}
\\
& \qquad\qquad\qquad \qquad\qquad\qquad\Big\{\Delta^2\varphi =1 \; \text{in $\Omega$}, 
\; \varphi=\frac{\partial \varphi }{\partial\nu }=0 \; \text{on $\partial\Omega$}\Big\}\,.
\end{split}
\end{equation}
(see \cite[Remark~2.1, p.~6]{chu-ryz}).
With these projections, we then set 
\begin{align}  \label{assign}
& w_1 =\mathbb{P}\hat{w}_1  \notag \\
& \pi_0 = q_0-\Delta^2(I-\mathbb{P})\hat{w}_1\,.
\end{align}
With this assignment of variables, then by \eqref{p_1} and $\hat{w}_{1}= 
\mathbb{P}\hat{w}_1+(I-\mathbb{P})\hat{w}_1$, we will have that $w_1$ solves 
\eqref{r2}--\eqref{r3}. (And of course since $\pi_0$ and $q_0$ differ only
by a constant, then the pair $(u,\pi_0)$ also solves \eqref{r4}--\eqref{r6-i}.)

Moreover, from elliptic theory, \eqref{stokes} and \eqref{plate_hat}, we
have the estimate 
\begin{equation}  \label{plate}
\begin{split}
& \|w_1\|_{H^4(\Omega )\cap H_0^2(\Omega )\cap L^2(\Omega)/\mathbb{R}}
+\|\pi_0\|_{H^1({\mathcal{O}})}\le  \notag \\[1mm]
& \qquad\qquad\qquad \le C\,\big(\|\Delta^2(I-\mathbb{P})\hat{w}%
_1\|_{L^2(\Omega)} +\|q_0\|_{H^1({\mathcal{O}})/\mathbb{R}%
}+\|w_2^*\|_{L^2(\Omega)}\big)  \notag \\[1mm]
& \qquad\qquad\qquad \le C\,\|[u^*,w_1^*,w_2^*]\|_{\mathbf{H}}\,,
\end{split}
\end{equation}
where implicitly we are also using the fact that $\Delta^2(I-\mathbb{P})\in {%
\mathcal{L}}(H_0^2(\Omega ),\mathbb{R})$, by the Closed Graph Theorem.

\medskip \noindent 
\emph{4. Resolution of the Pressure. } 
As we noted in \eqref{p_reg} we have $\Delta u\cdot \nu \in H^{-\frac{1}{2}}(\partial 
\mathcal{O})$, with the estimate 
\begin{equation} \label{v_tr}
\Vert \Delta u\cdot \nu |_{\Omega }\Vert _{H^{-1/2}(\partial {\mathcal{O}}
)} \le C\,\Vert u\Vert _{\mathbf{H}^{2}({\mathcal{O}})} 
\le C\,\big[\Vert u^{\ast }\Vert _{\mathcal{H}_{f}}+\Vert w_{1}^{\ast}\Vert _{H_{0}^{2}(\Omega )}\big]\,,  
\end{equation}
where for the second inequality we have also used (\ref{stokes}).

We will apply this estimate to the pressure variable $\pi _{0}$ in %
\eqref{e:explicit-resolvent-eq} 
-- given explicitly in \eqref{assign} -- which solves \textit{a fortiori}  
\begin{equation*} \label{abc}
\begin{cases}
\Delta \pi _{0}=0 & \text{in ${\mathcal{O}}$} 
\\[1mm] 
\frac{\partial \pi _{0}}{\partial \nu }=\Delta u\cdot \nu |_{S} & \text{on $S$}
\\[1mm]
\frac{\partial \pi _{0}}{\partial \nu }+\pi _{0}=\Delta ^{2}w_{1}+\Delta
u^{3}|_{\Omega } & \text{on $\Omega $} 
\end{cases}\quad.
\end{equation*}

In fact: Applying the divergence operator to both sides of \eqref{r4} and
using $\text{div}(\Delta u)=\text{div}(u^{\ast })=0$, we obtain that $\pi
_{0}$ is harmonic in ${\mathcal{O}}$. Moreover, dotting both sides of %
\eqref{r4} with repect to the normal vector, and subsequently taking the
boundary trace to the portion $S$, we get the boundary condition on $S$
(implicitly we are also using $u^{\ast }\cdot \nu |_{S}=0$, as $[u^{\ast
},w_{1}^{\ast },w_{2}^{\ast }]\in \mathbf{H}$). 
Finally, as $u^{\ast }\cdot \nu |_{\Omega }=w_{2}^{\ast }$, and as 
$[u^{\ast },w_{1}^{\ast },w_{2}^{\ast}]\in \mathbf{H}$, we have from \eqref{r2} 
\begin{equation*}
\pi _{0}|_{\Omega } =w_{2}^{\ast }+\Delta ^{2}w_{1} 
 =\Delta u\cdot \nu |_{\Omega }-\nabla \pi _{0}\cdot \nu |_{\Omega }+\Delta^{2}w_{1}\,,
\end{equation*}
which gives the boundary condition 
on $\Omega $ in \eqref{abc}. 
Necessarily then, the pressure term must be given by the expression 
\begin{equation} \label{d}
\pi _{0}=G_{0,1}(w_{1})+G_{0,2}(u)\in H^{1}({\mathcal{O}})
\end{equation}
(with the well-definition of right hand side assured by \eqref{v_tr}).

\medskip Finally, we collect: (i) \eqref{s1}--\eqref{s3-i} and 
\eqref{stokes} (for the fluid variable $u$); (ii) \eqref{p_1} and \eqref{r1}
and \eqref{assign} (for the respective structure and pressure variables 
$w_{1}$, $w_{2}$ and $\pi_{0}$); (iii) \eqref{plate} and \eqref{d}
(for the characterization of the pressure term $\pi _{0}$). In this way we
have obtained the solution of \eqref{e:explicit-resolvent-eq}-\eqref{pi_0}
in ${\mathcal{D}}(A)$. In short, $0\in \varrho (A)$, which concludes the proof.
\end{proof}


\subsection{Proof of the Main Result Theorem~\ref{t:main}}

\emph{Proof of Theorem~\ref{t:main}. } 
By Theorem \ref{t:pruess}, the fluid structure semigroup $\{e^{At}\}_{t\geq 0}$ will be uniformly stable provided its associated resolvent operator $R(\lambda ;A)$ is bounded on the
imaginary axis; viz., 
\begin{equation} \label{estimate}
\Vert R(i\beta ;A)\Vert _{\mathbf{H}}\le C\qquad \text{for all $\beta \in \mathbb{R}$}.  
\end{equation}

\smallskip
By way of establishing \eqref{estimate}, we consider the
following resolvent equation, for $\beta \in \mathbb{R}\setminus \{0\}$
(recall that we have already established that $0\in \varrho (A)$): Given data $%
[u^{\ast },w_{1}^{\ast },w_{2}^{\ast }]\in \mathbf{H}$, we look for $%
[w_{1},w_{2},u]\in {\mathcal{D}}(A)$ which solves 
\begin{equation}
(i\beta -A)\left[ 
\begin{array}{c}
u \\ 
w_{1} \\ 
w_{2}%
\end{array}%
\right] =\left[ 
\begin{array}{c}
u^{\ast } \\ 
w_{1}^{\ast } \\ 
w_{2}^{\ast }%
\end{array}%
\right] \,.  \label{resolve2}
\end{equation}%
(Here, we also take $|\beta |\geq 1$, as $|\beta |<1$ is relatively
straightfoward.)

From \eqref{domain} and \eqref{e:domain-of-generator} we see then that $%
[u,w_{1},w_{2}]\in {\mathcal{D}}(A)$ satisfies the following PDE system 
\begin{subequations}
\label{PDE}
\begin{align}
& i\beta u-\Delta u+\nabla p=u^{\ast }\in {\mathcal{H}}_{f} & &
\label{subpde-mechanical} \\
& \text{div}(u)=0 & & \text{in $\cO$} \\
& u=0 & & \text{on $S$} \\
& u=(u^{1},u^{2},u^{3})=(0,0,i\beta w_{1}-w_{1}^{\ast }) & & \text{on $%
\Omega $} \\
& i\beta w_{1}-w_{2}=w_{1}^{\ast }\in \big[H_{0}^{2}(\Omega )\cap
L^{2}(\Omega )/\mathbb{R}\big] & & \\
& -\beta ^{2}w_{1}+\Delta ^{2}w_{1}-p|_{\Omega }=w_{2}^{\ast }+i\beta
w_{1}^{\ast }\in L^{2}(\Omega )/\mathbb{R} & & \\
& w_{1}\big|_{\partial \Omega }=\frac{\partial w_{1}}{\partial n}\Big|%
_{\partial \Omega }=0\,. & &
\end{align}

\bigskip \noindent \emph{Step 1. (A relation for the fluid gradient)} We
start by taking the $\mathbf{H}$-inner product of both sides of %
\eqref{resolve2}, with respect to $[u,w_{1},w_{2}]$. This gives 
\end{subequations}
\begin{equation*}
i\beta \left\Vert \left[ 
\begin{array}{c}
u \\ 
w_{1} \\ 
w_{2}%
\end{array}%
\right] \right\Vert _{\mathbf{H}}^{2}-\left( A\left[ 
\begin{array}{c}
u \\ 
w_{1} \\ 
w_{2}%
\end{array}%
\right] \,,\left[ 
\begin{array}{c}
u \\ 
w_{1} \\ 
w_{2}%
\end{array}%
\right] \right) _{\mathbf{H}}=\left( \left[ 
\begin{array}{c}
u^{\ast } \\ 
w_{1}^{\ast } \\ 
w_{2}^{\ast }%
\end{array}%
\right] \,,\left[ 
\begin{array}{c}
u \\ 
w_{1} \\ 
w_{2}%
\end{array}%
\right] \right) _{\mathbf{H}}\,.
\end{equation*}%
Combining this with the readily derivable relation 
\begin{equation}
\left( A\left[ 
\begin{array}{c}
u \\ 
w_{1} \\ 
w_{2}%
\end{array}%
\right] \,,\left[ 
\begin{array}{c}
u \\ 
w_{1} \\ 
w_{2}%
\end{array}%
\right] \right) _{\mathbf{H}}=-\Vert \nabla u\Vert _{{\mathcal{O}}}^{2}-2i\,%
\text{Im}\big(\Delta w_{1},\Delta w_{2}\big)_{\Omega }
\label{e:dissipativity}
\end{equation}
(see \cite{avalos-bucci-arxiv2013}) we then will have the following
\textquotedblleft static dissipation\textquotedblright : 
\begin{equation}
\left\Vert \nabla u\right\Vert _{L^{2}(\mathcal{O})}^{2}=\text{Re}\left( %
\left[ 
\begin{array}{c}
u^{\ast } \\ 
w_{1}^{\ast } \\ 
w_{2}^{\ast }%
\end{array}%
\right] ,\left[ 
\begin{array}{c}
u \\ 
w_{1} \\ 
w_{2}%
\end{array}%
\right] \right) _{\mathbf{H}}\,.  \label{grad}
\end{equation}

This gives then, for arbitrary $\epsilon >0$, 
\begin{equation}  \label{est1}
\left\| \nabla u\right\|_{L^2({\mathcal{O}})}\le \epsilon \left\|\left[ 
\begin{array}{c}
u \\ 
w_1 \\ 
w_2%
\end{array}
\right] \right\|_{\mathbf{H}}+C_{\epsilon }\left\|\left[ 
\begin{array}{c}
u^* \\ 
w_1^* \\ 
w_2^*%
\end{array}
\right] \right\|_{\mathbf{H}}\,.
\end{equation}

\medskip \noindent 
\emph{Step 2. (Control of the mechanical velocity)} 
This comes quickly: using the fluid Dirichlet boundary condition in \eqref{PDE}
we have 
\begin{equation*}
i\beta w_{1}=u^{3}\big|_{\Omega }-w_{1}^{\ast }\,.
\end{equation*}%
We estimate this expression by invoking in sequence, the Sobolev Embedding
Theorem, Poincar\'{e} Inequality and \eqref{est1}. In this way, we then
obtain 
\begin{equation}
\begin{split}
\Vert \beta w_{1}\Vert _{H^{1/2}(\Omega )}& \leq \Vert \left.
u^{3}\right\vert _{\Omega }-w_{1}^{\ast }\Vert _{H^{1/2}(\Omega )} \\[1mm]
& \leq C\left( \,\Vert \nabla u\Vert _{L^{2}({\mathcal{O}})}+\Vert
w_{1}^{\ast }\Vert _{H_{0}^{2}(\Omega )}\right) \\[1mm]
& \leq \epsilon C\,\left\Vert \left[ 
\begin{array}{c}
u \\ 
w_{1} \\ 
w_{2}%
\end{array}%
\right] \right\Vert _{\mathbf{H}}+C_{\epsilon }\left\Vert \left[ 
\begin{array}{c}
u^{\ast } \\ 
w_{1}^{\ast } \\ 
w_{2}^{\ast }%
\end{array}%
\right] \right\Vert _{\mathbf{H}}\,.
\end{split}
\label{est2}
\end{equation}

Using subsequently the resolvent relation $w_{2}=i\beta w_{1}-w_{1}^{\ast }$
now gives 
\begin{equation}
\Vert \beta w_{1}\Vert _{H^{1/2}(\Omega )}+\Vert w_{2}\Vert _{H^{1/2}(\Omega
)}\leq \epsilon C\,\left\Vert \left[ 
\begin{array}{c}
u \\ 
w_{1} \\ 
w_{2}%
\end{array}%
\right] \right\Vert _{\mathbf{H}}+C_{\epsilon }\left\Vert \left[ 
\begin{array}{c}
u^{\ast } \\ 
w_{1}^{\ast } \\ 
w_{2}^{\ast }%
\end{array}%
\right] \right\Vert _{\mathbf{H}}\,.  \label{est3}
\end{equation}

\medskip \noindent 
\emph{Step 3. (Control of the mechanical displacement)}
We multiply both sides of the mechanical equation \eqref{subpde-mechanical}
by $w_1$ and integrate. This gives the relation 
\begin{equation}  \label{mech}
\left( \Delta^2 w_1,w_1\right)_{L^2(\Omega)}= \big(p|_{\Omega},w_1\big)%
_{\Omega }+\beta^2\left\| w_1\right\|_{L^2(\Omega)}^2 +\big(w_2^*+i\beta
w_1^*,w_1\big)_{L^2(\Omega)}\,.
\end{equation}
To handle the first term we use the fact that since $[w_1,w_2,u]\in \mathbf{H%
}$, then in particular 
\begin{equation*}
\int_{\Omega}w_1\, d\Omega =0\,.
\end{equation*}
In consequence, one has well-posedness of the following boundary value
problem (see \cite[Proposition 2.2]{temam}): 
\begin{equation}  \label{BVP}
\begin{cases}
-\Delta \psi +\nabla q=0 & \text{in ${\mathcal{O}}$} \\ 
\text{div}\psi =0 & \text{in ${\mathcal{O}}$} \\ 
\psi|_S=0 & \text{on $S$} \\ 
\psi|_\Omega=\big(\psi^1,\psi^2,\psi^3\big)\big|_{\Omega }=(0,0,w_1) & \text{%
on $\Omega$}%
\end{cases}%
\end{equation}
with the estimate 
\begin{equation}  \label{bvp_est}
\left\|\nabla \psi \right\|_{\mathbf{L}^2({\mathcal{O}})} +\|q\|_{L^2({%
\mathcal{O}})}\le C\,\|w_1\|_{H^{1/2}(\Omega)}
\end{equation}
(implicitly, we are also using Poincar\'e Inequality).

\smallskip With this solution variable $\psi $ in hand, we now address the
first term on the right hand side of \eqref{mech}: since normal vector $\nu $
equals $(0,0,1)$ on $\Omega $ (and as fluid variable $u$ is divergence
free), we have 
\begin{align}
(p|_{\Omega },w_{1})_{\Omega }& =-\left( \frac{\partial u}{\partial \nu },%
\left[ 
\begin{array}{c}
0 \\ 
0 \\ 
w_{1}%
\end{array}%
\right] \right) _{\mathbf{L}^{2}(\Omega )}+\left( p\,\nu ,\left[ 
\begin{array}{c}
0 \\ 
0 \\ 
w_{1}%
\end{array}%
\right] \right) _{\mathbf{L}^{2}(\Omega )}  \notag \\
& =-\left( \frac{\partial u}{\partial \nu },\psi \right) _{\mathbf{L}%
^{2}(\partial {\mathcal{O}})}+\left( p\,\nu ,\psi \right) _{\mathbf{L}%
^{2}(\partial {\mathcal{O}})}\,,  \label{p1}
\end{align}%
after invoking the boundary conditions in \eqref{BVP}.

The use of Green's Identities and the fluid equation in \eqref{p1} then
gives 
\begin{align*}
(p|_{\Omega },w_{1})_{\Omega }& =-\left( \frac{\partial u}{\partial \nu }%
,\psi \right) _{\mathbf{L}^{2}(\partial {\mathcal{O}})}+(p\,\nu ,\psi )_{%
\mathbf{L}^{2}(\partial {\mathcal{O}})} \\
& =-(\Delta u,\psi )_{\mathbf{L}^{2}({\mathcal{O}})}-\big(\nabla u,\nabla
\psi \big)_{\mathbf{L}^{2}({\mathcal{O}})}+\left( \nabla p,\psi \right) _{%
\mathbf{L}^{2}({\mathcal{O}})} \\
& =-i\beta (u,\psi )_{\mathbf{L}^{2}({\mathcal{O}})}-\left( \nabla u,\nabla
\psi \right) _{\mathbf{L}^{2}({\mathcal{O}})}+(u^{\ast },\psi )_{\mathbf{L}%
^{2}({\mathcal{O}})}\,.
\end{align*}%
Estimating the latter right hand side by means of \eqref{est1}, \eqref{est2}
and \eqref{bvp_est}, we get for $\left\vert \beta \right\vert >1,$%
\begin{align}
\left\Vert \left( p|_{\Omega },w_{1}\right) _{\Omega }\right\vert & \leq
C\,|\beta |\left\Vert \nabla \psi \right\Vert _{\mathbf{L}^{2}({\mathcal{O}}%
)}\,\left( \left\Vert \nabla u\right\Vert _{\mathbf{L}^{2}({\mathcal{O}}%
)}+\Vert u^{\ast }\Vert _{\mathbf{L}^{2}({\mathcal{O}})}\right)  \notag \\
& \leq \epsilon \,C\,\left\Vert \left[ 
\begin{array}{c}
u \\ 
w_{1} \\ 
w_{2}%
\end{array}%
\right] \right\Vert _{\mathbf{H}}+C_{\epsilon }\left\Vert \left[ 
\begin{array}{c}
u^{\ast } \\ 
w_{1}^{\ast } \\ 
w_{2}^{\ast }%
\end{array}%
\right] \right\Vert _{\mathbf{H}}\,.  \label{est4}
\end{align}

Applying the obtained estimate \eqref{est4} to the right hand side of %
\eqref{mech}, and using once more \eqref{est2}, we get 
\begin{equation*}
\left( \Delta ^{2}w_{1},w_{1}\right) _{L^{2}(\Omega )}\leq \epsilon
C\,\left\Vert \left[ 
\begin{array}{c}
u \\ 
w_{1} \\ 
w_{2}%
\end{array}%
\right] \right\Vert _{\mathbf{H}}+C_{\epsilon }\left\Vert \left[ 
\begin{array}{c}
u^{\ast } \\ 
w_{1}^{\ast } \\ 
w_{2}^{\ast }%
\end{array}%
\right] \right\Vert _{\mathbf{H}}\,.
\end{equation*}%
Another application of Green's formula, along the fact that $w_{1}$
satisfies hinged boundary conditions, gives 
\begin{equation}
\Vert \Delta w_{1}\Vert _{L^{2}(\Omega )}^{2}\leq \epsilon C\left\Vert \left[
\begin{array}{c}
u \\ 
w_{1} \\ 
w_{2}%
\end{array}%
\right] \right\Vert _{\mathbf{H}}+C_{\epsilon }\left\Vert \left[ 
\begin{array}{c}
u^{\ast } \\ 
w_{1}^{\ast } \\ 
w_{2}^{\ast }%
\end{array}%
\right] \right\Vert \,.  \label{est5}
\end{equation}

Thus, combining \eqref{est1}, \eqref{est3} and \eqref{est5} yields now 
\begin{equation*}
\left\Vert \lbrack u,w_{1},w_{2}]\right\Vert _{\mathbf{H}}^{2}\leq \epsilon
C\,\left\Vert [u,w_{1},w_{2}]\right\Vert _{\mathbf{H}}^{2}+C_{\epsilon
}\left\Vert [u^{\ast },w_{1}^{\ast },w_{2}^{\ast }]\right\Vert _{\mathbf{H}%
}^{2}\,.
\end{equation*}%
This yields the required uniform norm estimate \eqref{estimate}, upon taking 
$\epsilon >0$ small enough. The proof of exponential decay of finite energy
solutions of \eqref{e:pde-model}-\eqref{ic} is now complete. 

\qed


\section{The associated optimal control problems: relevant scenarios,
expected difficulties}

In this section we briefly discuss a couple of possible implementations for
the placement of control actions into the PDE system \eqref{e:pde-model}-%
\eqref{ic}; these are complemented with some remarks about the technical
challenges which are expected in the forthcoming study of the associated
optimal control problems (with quadratic functionals).

\subsection{A first setup: point control on the mechanical component}

A classical scenario worth studying is the case of \emph{point control}
exerted on the elastic wall $\Omega$. The control action may be
mathematically described by 
\begin{equation*}
{\mathcal{B}} g = \sum_{j=1}^J a_j\,g_j\delta_{\xi_j}\,,
\end{equation*}
where $\xi_j$ are points in $\Omega$, and $\delta_{\xi_j}$ denote the
corresponding \emph{delta functions}. The control space here is $U=\mathbb{R}%
^J$ and 
\begin{equation*}
{\mathcal{B}}: U \to H^{-1-\sigma}(\Omega)\,, \; \sigma>0\,;
\end{equation*}
accordingly, the initial/boundary value problem (IBVP) is as follows:

\begin{equation}  \label{e:control-pbm_1}
\begin{cases}
u_t-\Delta u +\nabla p= 0 & \quad\text{in }\; {\mathcal{O}}\times (0,T) \\ 
\text{div}\,u=0 & \quad\text{in }\; {\mathcal{O}}\times (0,T) \\ 
u=0 & \quad\text{on }\, S\times (0,T) \\ 
u=(u^1,u^2,u^3)=(0,0,w_t) & \quad\text{on }\, \Omega\times (0,T) \\ 
w_{tt} -\rho\Delta w_{tt} +\Delta^2 w = p|_{\Omega}+\boxed{\cB g} & \quad%
\text{in }\; \Omega\times (0,T) \\ 
w=\frac{\partial w}{\partial\nu}=0 & \quad\text{on }\, \partial\Omega\times
(0,T) \\ 
u(\cdot,0)=u_0 & \quad\text{in }\; {\mathcal{O}} \\ 
w(\cdot,0)=w_0\,, \; w_t(\cdot,0)=w_1 & \quad\text{in }\; \Omega\,.%
\end{cases}%
\end{equation}
To deal with the IBVP problem \eqref{e:control-pbm_1}, with the natural
state space $Y$ given by $\mathbf{H}_\rho$ and the control space $U$ defined
above, we will appeal to the the corresponding abstract formulation 
\begin{equation}  \label{e:abstract-system}
\begin{cases}
y^{\prime }(t)=Ay(t)+Bg(t)\,, & 0< t< T \\[1mm] 
y(0)=y_0\in Y\,, & 
\end{cases}%
\end{equation}
where

\begin{itemize}
\item $A:{\mathcal{D}}(A) \subset Y \to Y$ is the infinitesimal generator of
a $C_0$-semigroup $e^{At}$ on $Y$, $t\ge 0$;

\item $B\in {\mathcal{L}}(U,[{\mathcal{D}}(A^*)]^{\prime })$; equivalently, $%
A^{-1}B\in {\mathcal{L}}(U,Y)$.
\end{itemize}

The most prominent features of the controlled dynamics is that (i) the
control operator $B$ is \emph{not bounded} from $U$ into $Y$; (ii) the
semigroup $e^{At}$ is \emph{not analytic}; in fact the semigroup results
from a strong coupling of analytic Stokes flow with \emph{hyperbolic} plate
dynamics.

To the control system \eqref{e:abstract-system} we associate a quadratic
functional: 
\begin{equation} \label{e:functional-general}
J(g)=\int_{0}^{T}\big(\Vert Ry(t)\Vert _{Z}^{2}+\Vert g(t)\Vert _{U}^{2}\big)\,dt\,,  
\end{equation}
where $R\in {\mathcal{L}}(Y,Z)$ denotes the \emph{observation operator} and $Z$ the observation
space; thus, the optimal control problem is formulated as follows.

\begin{problem}[The optimal control problem]
Given $y_{0}\in Y$, we seek a control function $g\in L^{2}(0,T;U)$ which
minimizes the cost functional \eqref{e:functional-general}, where 
$y(\cdot)=y(\cdot \,;y_{0},g)$ is the solution to \eqref{e:abstract-system}
corresponding to $g(\cdot )$.
\end{problem}

As is well known, the core of the work is to pinpoint the regularity of the
input-to-state map 
\begin{equation*}
L:\,g(\cdot )\mapsto (Lg)(t):=\int_{0}^{t}e^{A(t-s)}Bg(s)\,ds\,,
\end{equation*}%
which, in turn, is related to the \emph{regularity properties} of the kernel 
$e^{At}B$ (or, equivalently, of $B^{\ast }e^{A^{\ast }t}$).

We expect to make use of the sharp regularity theory for the uncoupled plate
equation in the presence of point control; see e.g., \cite{triggiani-1993}. 
We also expect that, with the possible exception of the one-dimensional case
for $\Omega $, the presence of point control acting on the {\em hyperbolic}
component of the PDE system will prevent one from establishing that the \emph{gain operator}
is bounded -- unbounded even on a dense subset of $Y$ --, unless the
observation operator $R$ possesses appropriate smoothing properties.

\subsection{A different setup: boundary control on the fluid component}

Another interesting scenario is the case of \emph{boundary} control acting
on some part $\Sigma$ of $S$. A tentative condition\footnote{
This condition was suggested by Giovanna Guidoboni (IUPUI; also 
Acting Co-Director of the \emph{School of Science Institute of Mathematical
Modeling and Computational Science}, Indianapolis), 
in connection with the modeling of ocular blood flow and specifically with
the issue of reducing the ocular pressure.} to be taken into consideration
is $u\cdot \nu=g$ on $\Sigma$.
The IBVP becomes as follows: 
\begin{equation}
\begin{cases}
u_{t}-\Delta u+\nabla p=0 & \quad \text{in }\;{\mathcal{O}}\times (0,T) \\ 
\text{div}\,u=0 & \quad \text{in }\;{\mathcal{O}}\times (0,T) \\ 
u=0 & \quad \text{on }\,S\setminus \Sigma \times (0,T) \\ 
u\cdot \nu =\boxed{g}\,,\quad u\cdot \tau =0\,, & \quad \text{on }\,\Sigma
\times (0,T) \\ 
u=(u^{1},u^{2},u^{3})=(0,0,w_{t}) & \quad \text{on }\,\Omega \times (0,T) \\ 
w_{tt}-\rho \Delta w_{tt}+\Delta ^{2}w=p|_{\Omega } & \quad \text{in }%
\;\Omega \times (0,T) \\ 
w=\frac{\partial w}{\partial \nu }=0 & \quad \text{on }\,\partial \Omega
\times (0,T) \\ 
u(\cdot ,0)=u_{0} & \quad \text{in }\;{\mathcal{O}} \\ 
w(\cdot ,0)=w_{0}\,,\;w_{t}(\cdot ,0)=w_{1} & \quad \text{in }\;\Omega \,.%
\end{cases}
\label{e:control-pbm_2}
\end{equation}%
Notice that in the present case one has $U=L^{2}(\Sigma )$ and $B:g\mapsto
Bg=({\mathcal{B}}g,0,0)$, where ${\mathcal{B}}$ is an appropriate operator
which describes the particular normal component control action.
The presence of boundary control acting on the {\em parabolic} component of the PDE system
suggests that we investigate whether the optimal control theory of \cite{abl-2005, abl-2013} is applicable. The analyis to be performed will require
that appropriate regularity results for the boundary traces of the fluid
component on $\Sigma$ are established.

\begin{remark}
\begin{rm}
If the aforesaid theory applies, we will achieve well-posedness of
Riccati equations with \emph{bounded} gains (on a dense subset of $Y$),
without requiring smoothing effects of the observation operator $R$. 
This has been done successfully in the case of a different F-S model (solid
immersed in a fluid); see \cite{bucci-las-fs_1} and \cite{bucci-las-fs_2}.
(See \cite{las-tuff-1} and \cite{las-tuff-2} for a study of a Bolza problem associated 
with the very same PDE model.) 
\end{rm}
\end{remark}



\end{document}